\documentclass[12pt]{amsart}
\usepackage{amssymb, amsmath, mathtools}
\usepackage{url}

\newtheorem{theorem}{Theorem}[section]
\newtheorem{corollary}[theorem]{Corollary}
\newtheorem{prop}[theorem]{Proposition}

\theoremstyle{remark}

\theoremstyle{definition}
\newtheorem{defn}[theorem]{Definition}
\newtheorem{example}[theorem]{Example}

\numberwithin{equation}{section}
\numberwithin{theorem}{section}

\newcommand{\Psif}{\Psi_{\!f}}
\newcommand{\Phif}{\Phi_{\!f}}

\let\abs=\envert

\newcommand{\fhat}{\hat f}
\newcommand{\ghat}{\hat g}

\newcommand{\hhat}{\hat h}
\newcommand{\psihat}{\hat\psi}
\newcommand{\psiahat}{\widehat{\psi_a}}
\newcommand{\phihat}{\hat \phi}

\newcommand{\intinf}{\int^\infty_{-\infty}}

\newcommand{\N}{{\mathbb N}}
\newcommand{\R}{{\mathbb R}}
\newcommand{\C}{{\mathbb C}}

\newcommand{\Sc}{{\mathcal S}(\R)}

\newcommand{\fn}{\!:\!}

\providecommand{\abs}[1]{\lvert#1\rvert}
\providecommand{\norm}[1]{\lVert#1\rVert}
\providecommand{\Lany}[1]{L^{#1}(\R)}
\newcommand{\fonehat}{\hat{f_1}}
\newcommand{\ftwohat}{\hat{f_2}}
\newcommand{\intin}{\int_{-1}^1}
\newcommand{\intout}{\int_{\abs{t}>1}}

\begin{document}


\date{Preprint December 12, 2025.}

\title[Fourier transforms of bounded functions]{Fourier transforms of bounded functions}
\author{Erik Talvila}
\address{Department of Mathematics \& Statistics\\
University of the Fraser Valley\\
Abbotsford, BC Canada V2S 7M8}
\email{Erik.Talvila@ufv.ca}

\begin{abstract}
The Fourier transform of a bounded measurable function, $f$, on the real line is shown to be
the second distributional derivative
of a H\"older continuous function.  The Fourier
transform is written as the difference of $\int_{-1}^1 e^{-ist}f(t)\,dt$ and
the second distributional derivative of the integral 
$\int_{\lvert{t}\rvert>1}e^{-ist}f(t)\,dt/t^2$. The 
space of such Fourier transforms is isometrically isomorphic
to $L^\infty(\mathbb{R})$.
There is an exchange theorem, inversion
and convolution results.  The Fourier transform of the functions
$x\mapsto\cos^m(a/x)$
for each natural number $m$ are computed.  Also for 
$x\mapsto x\sin(a/x)$ and $x\mapsto\arctan(x/a)$.
\end{abstract}

\maketitle

\section{Introduction}\label{sectionintroduction}

If $f$ is a bounded and Lebesgue measurable function on the real line then the 
Fourier transform is  shown to be the second distributional derivative of
a H\"older continuous function.  The Fourier transform is written as the difference of 
the integrals
$\intin e^{-ist}f(t)\,dt$ and $\intout e^{-ist}f(t)\,dt/t^2$.  The first integral
exists as usual. The second is differentiated twice with the distributional derivative.
It is shown that this second integral is  H\"older continuous but need not be
Lipschitz continuous.  This formulation gives more information about the Fourier
transform than the exchange formula definition for tempered distributions, 
$\langle \fhat,\phi\rangle=\langle f,\phihat\rangle$,
where $\phi$ is a test function.  This leads to an exchange formula valid when
$\phi'$ need only be of bounded variation.  Inversion in
norm with a summability kernel follows as well as analogues of the usual convolution
results.  In this case, the Fourier transform need not have any pointwise values.
It is integrable in the sense of distributional integrals, which convert to
Stieltjes integrals.

If $f\in L^1(\R)$ then its Fourier transform is $\fhat(s)=\intinf e^{-ist}f(t)\,dt$
for $s\in\R$. The Fourier transform has the following well known properties:
\begin{align}
&\norm{\fhat}_\infty\leq\norm{f}_1\label{FT1}\\
&\fhat \text{ is uniformly continuous on } \R \label{FT2}\\
&\lim_{\abs{s}\to\infty}\fhat(s)=0 \text{ (Riemann--Lebesgue lemma). }\label{FT3}
\end{align}

Background on Fourier transforms in $L^1(\R)$ can be found, for example, in \cite{folland,
grafakosclassical, steinweiss}.

If $f\in L^p(\R)$ for some $1<p\leq 2$ then the defining integral need not exist but
the Fourier transform can be defined using interpolation methods 
\cite{folland, grafakosclassical,
steinweiss}.  It then happens that the Fourier transform maps $L^p(\R)$ into $L^q(\R)$,
where $1/p+1/q=1$.  In the case $p=2$ the Fourier transform is an isometric
isometry on $L^2(\R)$.

If $f\in L^p(\R)$ for some $2<p\leq\infty$ then interpolation methods do not apply.
The Fourier transform can be defined as a tempered distribution via the exchange
formula $\langle \fhat,\phi\rangle=\langle f,\phihat\rangle$, for all test functions
$\phi\in\Sc$, where the Schwartz space of test functions consists 
of the $C^\infty(\R)$ functions such
that, for all integers $m,n\geq 0$, $x^m\phi^{(n)}(x)\to0$ as $\abs{x}\to\infty$.
See, for example, \cite{donoghue, folland, friedlanderjoshi,
steinweiss}.
In this case the Fourier transform is a tempered distribution but need not be a function \cite[4.13, p.~34]{steinweiss}
and the Fourier transform need not be a bounded linear operator to 
any $L^p(\R)$ space \cite{abdelhakim}.
In \cite{talvilaLpFourier} it was shown for $1<p<\infty$ that the Fourier transform is the distributional
derivative of a H\"older continuous function.

In this paper, Fourier transforms are defined for bounded measurable functions using the
second distributional derivative.  If $f\in L^\infty(\R)$ then write $f=f_1+f_2$ where
$f_1=f\chi_{[-1,1]}$ and $f_2=f\chi_{\R\setminus[-1,1]}$.  Then 
$f_1\in L^1(\R)$ and $\fonehat$ is defined with the usual integral.
Notice that the integral
\begin{equation}
\Psif(s)=\intout e^{-ist}f(t)\frac{dt}{t^2}\label{Psifintegral}
\end{equation}
exists.  The second distributional derivative is used to define $\ftwohat=-\Psif''$.  
Thus, $\langle \Psif'',\phi\rangle=\langle\Psif,\phi''\rangle=\intinf\Psif(s)\phi''(s)\,ds$
for all $\phi\in\Sc$.  The Fourier transform is then $\fhat=\fonehat+\ftwohat$.
If $f\in L^1(\R)$, by dominated convergence, the second pointwise derivative gives
$\Psif''(s)=\ftwohat(s)$ for each $s\in\R$.

The outline of the paper is as follows.

Properties of the function $\Psif$ are
given in Theorem~\ref{theoremPsif}, including pointwise growth estimates.  It is
shown that $\Psif$ is H\"older continuous but need not be Lipschitz continuous.
Since the Fourier transform of $f\chi_{[-1,1]}$ is analytic, the Fourier transform
of $f$ is then the second distributional derivative of a H\"older continuous function
$\Phif$
(Theorem~\ref{theoremFT}).  

In Theorem~\ref{theorembanach}, norms are introduced so
that $L^\infty(\R)$, the space of Fourier transforms of $L^\infty(\R)$ functions,
and the space of all $\Phif$, are all isometrically isomorphic.  

Fourier transforms
of functions in $L^\infty(\R)$ are distributions that can be integrated on $\R$ when
multiplied by functions whose derivative is of bounded variation.  The integral of
such a distribution (second derivative of a continuous function)  is defined as a 
second order distributional integral in \cite{talvilaacrn} and this reduces to a
first order distributional integral \cite{talviladenjoy}, which in turn can
be evaluated in terms of Stieltjes integrals.  See Proposition~\ref{propA2}.
The exchange formula $\langle\fhat,g\rangle=\langle f,\ghat\rangle$, 
Theorem~\ref{theoremexchange},
then holds for $f\in L^\infty(\R)$ and functions $g$ whose Fourier transform is in
$L^1(\R)$.  Sufficient conditions are given for this to hold, essentially that
$g'$ be of bounded variation.  This then extends the exchange definition of the
Fourier transform of a tempered distribution, for which $g$ is taken to be a
Schwartz function, i.e., much smoother with more rapid decay at infinity.

Various examples are given in Section~\ref{sectionexamples}., including a function
$f\in L^\infty(\R)$ such that the integral $\intinf e^{-ist}f(t)\,dt$ diverges 
for each $s\in\R$.  Fourier transforms of the functions $x\mapsto\cos^m(a/x)$
for each $m\in\N$ are computed.  Similarly for $x\mapsto x\sin(a/x)$ and $x\mapsto\arctan(x/a)$.

An inversion theorem,
using convolution with a family of kernels whose Fourier transforms
are in $L^1(\R)$, is given in Theorem~\ref{theoreminversion}.  This leads to pointwise
inversion almost everywhere, or within the norm $\norm{\cdot}_\infty$,
if $f$ is also uniformly continuous.

In Section~\ref{sectionconvolution},
weak type convolution results $\intinf \widehat{f\ast g}h=\intinf\fhat\ghat\,h$
and $\intinf \fhat g\ast h=\intinf f(t)\ghat(t)\hhat(t)\,dt$ are proved for
$f\in L^\infty(\R)$ and
modest conditions on the functions $g$ and $h$.

Functions of bounded variation are used in Sections~\ref{sectionexchange}, \ref{sectioninversion}
and \ref{sectionconvolution}.  Notice that if $g$ is absolutely continuous then the variation
of $g$ is given by ${\rm var}(g)=\intinf\abs{g'(t)}\,dt$.  Note that functions of
bounded variation are bounded.  For background on functions
of bounded variation see, for example, \cite{appell}.

\section{The Fourier transform}

The key to the definition of the Fourier transform of a bounded function is the function
$\Psif$.  It's growth and smoothness properties are investigated in the first theorem.

A function $g\fn\R\to\R$ is H\"older continuous on $\R$ with exponent $0<\alpha\leq 1$ if
there is a constant $k$ such that
$\abs{f(x+h)-f(x)}\leq k\abs{h}^\alpha$ for all $h,x\in\R$.  If $\alpha=1$ the function is
Lipschitz continuous; which entails that $f$ is absolutely continuous, the pointwise 
derivative exists almost everywhere and is bounded.

\begin{theorem}\label{theoremPsif}
Let $f\in L^\infty(\R)$.  
Let $F(x)=f(x)\chi_{\R\setminus[-1,1]}/x^2$.
Let $\Psif(s)=\hat{F}(s)$.
Then $\Psif$ has the following properties.\\
(a) $\norm{\Psif}_\infty\leq\norm{F}_1$,\\
(b) $F\in L^r(\R)$ for each $1/2<r\leq\infty$,
$\Psif\in L^q(\R)$ for each $2\leq q\leq\infty$ and there is a constant $C_p$
such that $\norm{\Psif}_q\leq C_p\norm{F}_p$ where $1/p+1/q=1$,\\
(c) $\lim_{\abs{s}\to\infty}\Psif(s)=0$,\\
(d) $\Psif$ is uniformly continuous on $\R$.\\
(e) Let $h\in\R$ such that $0<\abs{h}<1/e$.  Then $\abs{\Psif(s+h)-\Psif(s)}
\leq 6\norm{f}_\infty\abs{h}\abs{\log\abs{h}}$ for all $s\in\R$.  It follows
that $\Psif$ is H\"older continuous for each positive exponent less than one.\\
(f) For each $s\in\R$ the estimate in (e) is sharp in the sense that if $\psi\fn\R\to(0,\infty)$
and $\psi(h)=o(h\log\abs{h})$ as $h\to0$ then there is $f\in L^\infty(\R)$ such
that $(\Psif(s+h)-\Psif(s))/\psi(h)$ is not bounded as $h\to0$.  It follows
that $\Psif$ need not be Lipschitz continuous.
\end{theorem}

\begin{proof}
Parts (a), (c) and (d) follow as in \eqref{FT1}, \eqref{FT2} and \eqref{FT3}.

The Hausdorff--Young--Babenko--Beckner inequality gives (b).
For the value of the constant $C_p$, see \cite{liebloss}.
Note that $C_1=1$.

To prove (e) note that
\begin{eqnarray*}
\abs{\Psif(s+h)-\Psif(s)} & = & \left|\intout e^{-ist}\left(1-e^{-iht}\right)f(t)\frac{dt}{t^2}
\right|\leq
2\norm{f}_\infty\abs{h}\int_{\abs{h}/2}^\infty\abs{\sin(t)}\frac{dt}{t^2}\\
 & \leq &  2\norm{f}_\infty\abs{h}\left(\int_{\abs{h}/2}^1 \frac{dt}{t} +\int_1^\infty\frac{dt}{t^2}
\right)
\leq 6\norm{f}_\infty\abs{h}\abs{\log\abs{h}}.
\end{eqnarray*}

For (f), define the family of linear operators $L_{s,h}\fn L^\infty(\R)\to \C$ by
$L_{s,h}[f]=(\Psif(s+h)-\Psif(s))/\psi(h)$.  The estimate in (e) shows $L_{s,h}$ is
a bounded linear operator for each $s,h\in\R$.  Let 
$f_{s,h}(x)=e^{isx}(1-e^{ihx})/\abs{1-e^{ihx}}$ with $f_{s,h}(x)=e^{isx}$ if $hx=0$.  
Then $\norm{f_{s,h}}_\infty=1$.  And,
$$
\frac{L_{s,h}[f_{s,h}]}{\norm{f_{s,h}}_\infty}  =  \frac{2\abs{h}}{\psi(h)}\int_{\abs{h}/2}
^\infty\abs{\sin(t)}
\frac{dt}{t^2}\geq \frac{4\abs{h}}{\pi\psi(h)}\int_{\abs{h}/2}^{\pi/2}\frac{dt}{t}\geq
\frac{\abs{h}\abs{\log\abs{h}}}{\pi\psi(h)}.
$$
This is not bounded as $h\to0$ so by the Uniform Boundedness Principle there is $f\in L^\infty(\R)$
such that $(\Psif(s+h)-\Psif(s))/\psi(h)$ is not bounded.  It follows that the growth condition is
sharp and $\Psif$ need not be Lipschitz continuous.
\end{proof}

The Fourier transform on the interval $[-1,1]$ is defined in the usual way since
if $f\in L^\infty(\R)$ then $f\in L^1([-1,1])$.  There is a corresponding function
whose second derivative gives $\fonehat$.

\begin{theorem}\label{theoremf1}
Let $f\in L^\infty([-1,1])$.\\
(a) Then $\norm{\fhat}_\infty\leq \norm{f}_1$.\\
(b) For each $1\leq p\leq 2$, $f\in L^p([-1,1])$ and $\norm{\fhat}_q\leq C_p\norm{f}_p$,
where $C_p$ is the constant in the Hausdorff--Young--Babenko--Beckner inequality,
and $q$ is conjugate to $p$.\\
(c) The Riemann--Lebesgue lemma holds, $\lim_{\abs{s}\to\infty}\fhat(s)=0$.\\
(d) $\fhat$ is analytic on $\C$.\\
(e) Let $\Omega_f(s)=\int_{0}^s\int_{0}^{\sigma_2}\fhat(\sigma_1)\,d\sigma_1\,d
\sigma_2$.  Let $v_s(t)=(1-ist-e^{-ist})/t^2$ with $v_s(0)=s^2/2$.  Then $\Omega_f(s)
=\int_{-1}^1v_s(t)f(t)\,dt$, $\Omega_f$ is analytic on $\C$ and 
$\fhat(s)=\Omega_f''(s)$ for each $s\in\R$.\\
(f) The estimate $\Omega_f(s)=O(s\log\abs{s})$, as $\abs{s}\to\infty$, is sharp in 
the sense that $\Omega_f(s)$ need not be $o(s\log\abs{s})$ as $\abs{s}\to\infty$.\\
(g) For each $s,h\in\R$,
\begin{eqnarray*}
\abs{\Omega_f(s+h)-\Omega_f(s)} & \leq &  \left\{\begin{array}{cl}
\norm{f}_\infty(2\abs{h}\abs{s}+h^2), & 0\leq \abs{s}\leq 1,\\
\norm{f}_\infty(2\abs{h}+4\abs{h}\log\abs{s}+h^2), & \abs{s}\geq 1.\\
\end{array}
\right.
\end{eqnarray*}
\end{theorem}
\begin{proof}
Parts (a) and (c) follow from \eqref{FT1}, \eqref{FT2} and \eqref{FT3}.

Part (b) as per (b) in Theorem~\ref{theoremPsif}.

Part (d) follows from dominated convergence.

In (e), use the Fubini--Tonelli theorem and then dominated convergence or
the fundamental theorem of calculus.  

To prove (f) note, 
let $s>1$ and
integrate to get
\begin{eqnarray*}
\abs{\Omega_f(s)} & \leq & s\int_{-s}^{s}\abs{v_1(t)}\abs{f(t/s)}\,dt\\
 & \leq & \norm{f}_\infty s\left(\int_{-1}^{1}\abs{v_1(t)}\,dt 
+2\int_1^s\left(\frac{2}{t^2}+\frac{1}{t}\right)dt\right)\\
 & \leq & \norm{f}_\infty s\left(\int_{-1}^1\abs{v_1(t)}\,dt+4
\left(1-\frac{1}{s}\right)+2\log(s)\right)\\
 & = & O(s\log(s)) \text{ as } s\to\infty.
\end{eqnarray*}
Similarly, as $s\to-\infty$.

To show this is sharp, let $f(x)={\rm sgn}(x)$.  If $s>1$ then
$$
\Omega_f(s)=-2is\left(\int_0^1\frac{t-\sin(t)}{t^2}dt+\int_1^s\frac{t-\sin(t)}{t^2}dt\right).
$$  
The integral $\int_0^1\frac{t-\sin(t)}{t^2}dt$ converges and
$\int_1^s\frac{t-\sin(t)}{t^2}dt\geq \int_1^s\frac{t-1}{t^2}dt\sim\log{s}$ as $s\to\infty$.

For (g), by Taylor's theorem there is $\xi$ between $s$ and $s+h$ so that
$$
\Omega_f(s+h)-\Omega_f(s)=\intin\frac{1-e^{-ist}}{it}f(t)\,dt\,h +\fhat(\xi)h^2/2.
$$
H\"older's inequality then shows
$$
\abs{\Omega_f(s+h)-\Omega_f(s)}  \leq 
4\norm{f}_\infty\int_0^{\abs{s}}\abs{\sin(t/2)}\frac{dt}{t}\abs{h}+\norm{f}_\infty h^2.
$$
Now use the method in the proof of (f).
\end{proof}

\begin{defn}\label{defnFT}
Let $f\in L^\infty(\R)$.  Write $f=f_1+f_2$ where $f_1=f\chi_{[-1,1]}$ and 
$f_2=f\chi_{\R\setminus[-1,1]}$.
Define $\fonehat(s)=\intin e^{-ist}f(t)\,dt$.
Let $\Psif(s)=\intout e^{-ist}f(t)\frac{dt}{t^2}$.  Define the Fourier transform of
$f_2$ by $\ftwohat=-\Psif''$, i.e., $\langle\ftwohat,\phi\rangle=-\langle\Psif'',\phi\rangle
=-\intinf\Psif(s)\phi''(s)\,ds$ for each $\phi\in\Sc$.
Define $\fhat=\fonehat+\ftwohat$.
\end{defn}

The results of Theorem~\ref{theoremPsif} and Theorem~\ref{theoremf1} give basic properties of
the Fourier transform.
\begin{theorem}\label{theoremFT}
Let $f\in L^\infty(\R)$. With Definition~\ref{defnFT}, there is a
continuous function $\Phif$ such that $\fhat=\Phif''$ and a constant $k$  
such that $\abs{\Phif(s+h)-\Phif(s)}
\leq k\norm{f}_\infty(\abs{h}\abs{\log\abs{h}}+\log(\abs{s}+1))$ for 
each $s\in\R$ and $0<\abs{h}\leq 1/e$.
The function $\Phif$ need not be Lipschitz continuous.
\end{theorem}
Note that the theorem says $\Phif$ is H\"older continous on compact sets, for each exponent
$\alpha$.

\section{Banach spaces}

The definition $\fhat=-\Psif''$ agrees with the tempered distribution definition,
$\langle \fhat,\phi\rangle=\langle f,\phihat\rangle$ for all test functions $\phi$.
This sets up isomorphisms between $L^\infty(\R)$, the set of functions $\Phif$ and
the space of Fourier transforms.

\begin{theorem}\label{theorembanach}
Let $f\in L^\infty(\R)$.
Using the notation of Theorem~\ref{theoremPsif} and Definition~\ref{defnFT} let
$\Phif=\Omega_f-\Psif$.  Write the Fourier transform as
$\fhat=\mathcal{F}[f]=\Omega_f''-\Psif''$.\\
(a) For each $\phi\in\Sc$, $\langle\fhat,\phi\rangle=\langle f,\phihat\rangle$.\\
(b) Let $\mathcal{A}_\infty(\R)=\{\fhat\mid f\in L^\infty(\R)\}$.
Define a norm on $\mathcal{A}_\infty(\R)$ by $\norm{\fhat}_{\mathcal{A}}=\norm{f}_\infty$.
Then $\mathcal{F}\fn L^\infty(\R)\to\mathcal{A}_\infty(\R)$ is an isometric isomorphism.\\
(c) Let $\mathcal{B}_\infty(\R)=\{\Phif\mid f\in L^\infty(\R)\}$.
Define a norm on $\mathcal{B}_\infty(\R)$ by $\norm{\Phif}_{\mathcal{B}}=\norm{f}_\infty$.
Then the three Banach spaces $L^\infty(\R)$, $\mathcal{A}_\infty(\R)$ and 
$\mathcal{B}_\infty(\R)$ are isometrically isomorphic.
\end{theorem}

\begin{proof}
(a) Note that $v_s(t)=s^2v_1(st)$ so that $\norm{v_s}_\infty=s^2\norm{v_1}_\infty$.
Since $\partial^2v_s(t)/\partial s^2=e^{-ist}$, the Fubini--Tonelli theorem and
integration by parts shows $\langle \Omega_f'',\phi\rangle=\langle f_1,\phihat\rangle$
for all test functions $\phi$.  Similarly, 
$\langle \Psif'',\phi\rangle=-\langle f_2,\phihat\rangle$.  Hence, $\langle \fhat,\phi\rangle
=\langle f,\phihat\rangle$ for all $\phi\in\Sc$.

(b)
Clearly, $\mathcal{F}$ is linear.  By definition it is onto $\mathcal{A}_\infty(\R)$.
To show $\mathcal{F}$ is one-to-one, suppose $f\in L^\infty(\R)$ and $\fhat=0$.
By (a), 
$\langle f,\phihat\rangle=0$ for all $\phi\in\Sc$.  Since the Fourier transform is an
isomorphism on $\Sc$ it follows that $f=0$ as a tempered distribution and the
Fourier transform is an isometric isomorphism.

(c) The second distributional derivative provides a linear mapping of $\mathcal{B}_\infty(\R)$
onto $\mathcal{A}_\infty(\R)$. 
If $\Phi_f''=0$ then (b) shows $f=0$ and hence $\Phi_0=0$.
\end{proof}

Since the definition of $\fhat$ agrees with that for the Fourier transform of a tempered
distribution it has the usual distributional
properties.  For example, \cite[\S30]{donoghue} or \cite[2.3.22]{grafakosclassical}.

\section{Integration and exchange theorem}\label{sectionexchange}

The exchange formula, $\langle\hat{T},\phi\rangle =\langle T,\phihat\rangle$,
is the definition of the Fourier transform of a tempered distribution $T$,
where $\phi$ is a test function in $\Sc$.  For Fourier transforms of bounded functions
this formula holds for a wider class of test functions, whose derivative
is merely of bounded variation.

In Definition~\ref{defnFT} and Theorem~\ref{theoremPsif} it was seen that
$\ftwohat$ was the second distributional derivative
of the H\"older continuous function $\Psif$.  Because $\Psif$ has limit $0$ at
$\pm\infty$ it is continuous
on the extended real line.  The integral of $\ftwohat$ then exists in the
sense of second distributional integrals \cite{talvilaacrn}.
According to this theory, the second distributional derivative of $\Psif$ can
be integrated against a function whose derivative is of bounded variation.
Write $g(x)=\int_0^x h(t)\,dt$ where $h$ is a right continuous function of bounded
variation.  Then
\begin{eqnarray}
\intinf \Psif''g & = & -\intinf\Psif'h=-\lim_{t\to\infty}\Psif(t)h(t)+\lim_{t\to-\infty}
\Psif(t)h(t)+\intinf\Psif(t)\,dh(t)\label{Ac2integral}\\
 & = & \intinf\Psif(t)\,dh(t).\label{HKStieltjesintegral}
\end{eqnarray}
Since $h$ is bounded and $\Psif$ vanishes at $\pm\infty$ (Theorem~\ref{theoremPsif}(c)),
the limits in \eqref{Ac2integral} vanish.
The multipliers
are then functions that are the indefinite integral of a function of bounded variation.
The above integration
by parts formula gives the integral of $\Psif''$ in terms of a continuous primitive
integral \cite{talviladenjoy} (second integral in \eqref{Ac2integral}), 
which is in turn defined as a Stieltjes integral 
\eqref{HKStieltjesintegral}
\cite[p.~199]{mcleod}.   

The Fourier transform of $f_1$ is continuous so it is
locally Lebesgue integrable.

To distinguish between different types of integrals, Lebesgue integrals will explicitly
show the integration variable and differential in the integrand, as in 
\eqref{Psifintegral} and
\eqref{HKStieltjesintegral},
 and distributional
integrals will not show these, as in \eqref{Ac2integral}.

\begin{prop}\label{prop}\label{propA2}
Let $f\in L^\infty(\R)$.
Let $h$ be of bounded variation and right continuous.  Define
$g(x)=\int_0^xh(t)\,dt$.  Use the notation of Definition~\ref{defnFT}.  
Then $\intinf \ftwohat g=-\intinf\Psif'h=\intinf\Psif(t)\,dh(t)$.
\end{prop}
\begin{proof}
See Definition~2.6 in \cite{talvilaacrn} and Definition~6 in \cite{talviladenjoy}.
\end{proof}

This integration formula leads to the exchange theorem, which is used to prove the
inversion theorem and integration of convolutions below.

\begin{theorem}[Exchange]\label{theoremexchange}
Let $f\in L^\infty(\R)$.  Let ($\alpha$) $g$ be absolutely continuous such that,
($\beta$) $g\in L^1(\R)$, ($\gamma$) $g'\in L^1(\R)$ and ($\delta$) $g'$ is of bounded variation.  Then $\ghat\in L^1(\R)$.
Using notation from Theorem~\ref{theoremPsif} and Definition~\ref{defnFT},
$\langle \fhat,g\rangle=\langle f,\ghat\rangle$.
\end{theorem}
\begin{proof}
The conditions ($\alpha$), ($\beta$) and ($\gamma$) imply $g$ is of bounded variation and 
$\lim_{\abs{s}\to\infty}g(s)=0$.  See also \cite[Lemma~2]{talvilafourieramm}.
There is a right continuous function $h\in L^1(\R)$ of bounded variation such that
$g(x)=\int_{-\infty}^xh(t)\,dt$.  It follows that $\lim_{\abs{s}\to\infty}h(s)=0$.
Since $h$ is right continuous, $g'$ can be defined at each point as $g'(x)=h(x+)$.
Proposition~7.2(c) in \cite{talvilaLpFourier} shows that $\ghat\in L^1(\R)$.

The formula $\intinf \fonehat(s)g(s)\,ds=\intin f(t)\ghat(t)\,dt$ holds
for all $g\in L^1(\R)$.

Using Proposition~\ref{propA2}, integrate to get
$$
\intinf \ftwohat g  =  -\intinf\Psif(s)\,dh(s) = -\intinf\intout e^{-ist}f(t)\frac{dt}{t^2}\,dh(s).
$$
Apply the Fubini--Tonelli theorem and integrate by parts twice to get
$$
\intinf \ftwohat g =\intout f(t)\ghat(t)\,dt.
$$
Then $\langle \fhat,g\rangle=\langle f,\ghat\rangle$.
\end{proof}
\begin{corollary}
Using notation from the proof of Theorem~\ref{theoremexchange},
$$
\abs{\langle \fhat, g\rangle}\leq \norm{f}_\infty\norm{\ghat}_1\leq
2\norm{f}_\infty(\norm{g}_1+{\rm var}(h)).
$$
\end{corollary}
\begin{proof}
For $s\not=0$ integrating by parts twice gives
$\ghat(s)=-\frac{1}{s^2}\intinf e^{-ist}\,dh(t)$. It follows that
$\norm{\ghat}_1\leq 2\int_0^1\norm{g}_1\,ds+2\int_1^\infty {\rm var}(h)\,ds/s^2$.
The result then follows.
\end{proof}

\section{Examples}\label{sectionexamples}

\begin{example}
Let $p$ be any measurable function on $\R$.  Define $f\in L^\infty(\R)$ by
$f(x)=e^{ip(x)}$.  Then $\Psif(s)=\intout e^{-ist}e^{ip(t)}\frac{dt}{t^2}$
and $\abs{\Psif(s)}\leq 2$.

Fourier transforms such as $\int_0^\infty e^{-ist}e^{it^\alpha}\,dt$ for $\alpha\in\R$
are then included in this rubric.
\end{example}

\begin{example}
Let $f(t)=e^{ixt}$ for $x\in\R$.  Then $\Psif(s)=2\abs{s-x}\int_{\abs{s-x}}^\infty
\cos(t)\,dt/t^2$.  It follows that $\Psif$ is absolutely continuous.
The almost everywhere pointwise derivative then coincides with the distributional
derivative.  For $s\not=x$ we have, after integration by parts, 
\begin{eqnarray*}
\Psif'(s) & = &  2\,{\rm sgn}(s-x)\int_{\abs{s-x}}^\infty\cos(t)\frac{dt}{t^2}
-\frac{2\cos(s-x)}{s-x}
 =  -2\,{\rm sgn}(s-x)\int_{\abs{s-x}}^\infty\sin(t)\frac{dt}{t}\\
 & = & -\pi\,{\rm sgn}(s-x)+2\int_0^1\sin[(s-x)t]\frac{dt}{t}.
\end{eqnarray*}

The second distributional derivative is then given using the Dirac distribution supported at $x$,
$$
\Psif''(s)=-2\pi\delta(s-x)+2\int_0^1\cos((s-x)t)\,dt=-2\pi\delta(s-x)+\frac{2\sin(s-x)}{s-x}.
$$
And, $\fonehat(s)=\intin e^{-ist}e^{ixt}\,dt=2\sin(s-x)/(s-x)$ so that 
$\fhat=\fonehat-\Psif''=
2\pi\delta_x$, in agreement with the usual formula.

Now let $g$ satisfy conditions ($\alpha$)-($\delta$) of Theorem~\ref{theoremexchange}.  Then
$\intinf\fhat g=\intinf f(t)\ghat(t)\,dt=2\pi\intinf\delta_xg$.  Since $g$ is
continuous, $g(x)=(2\pi)^{-1}\intinf e^{ixt}\ghat(t)\,dt$.  This then reproduces
the  usual inversion formula \eqref{inversionformula} for $L^1(\R)$ functions 
but with sufficient conditions
for $\ghat$ to be in $L^1(\R)$.
\end{example}

\begin{example}\label{examplekolmogorov}
Kolmogorov gave an example of a periodic $L^1$ function for which the inversion
integral diverges everywhere \cite{kolmogorov1926}.  Due to remarks on transference principles
\cite[p.~481]{grafakosmodern} there is a function $g\in L^1(\R)$ such that
$\lim_{R\to\infty}\int_{-R}^R e^{ixs}\ghat(s)\,ds$ diverges for each $x\in\R$.
Let $f=\ghat\in L^\infty(\R)$.  There is thus a uniformly continuous bounded
function $f$, with limits $0$ at $\pm\infty$, such 
that $\intinf e^{-ist}f(t)\,dt$ diverges for each $s\in\R$.
Definition~\ref{defnFT} then properly extends the notion of Fourier transform
beyond what the usual integral gives.
\end{example}

\begin{example}\label{examplecos}
Define $f(x)=\cos(a/x)$ for $a\not=0$.  Then $f\in L^\infty(\R)$ but $f$ is not
in any other $L^p(\R)$ space.  This function is also not of bounded variation.
The Fourier transform integral diverges when 
considered as a Henstock--Kurzweil integral or
distributional integral \cite{talvilaacrn,  talvilaLp}.  We have
$$
\Psif(s)=\intout e^{-ist}\cos\left(\frac{a}
{t}\right)\frac{dt}{t^2}=2\int_1^\infty\cos(st)\cos\left(\frac{a}
{t}\right)\frac{dt}{t^2}.
$$
The derived integral converges uniformly for all $\abs{s}\geq\eta>0$ so that
\begin{eqnarray}
\Psif'(s) & = & -2\int_1^\infty\sin(st)\cos\left(\frac{a}
{t}\right)\frac{dt}{t}\notag\\
 & = & -2\int_0^\infty\sin(st)\cos\left(\frac{a}
{t}\right)\frac{dt}{t} +2\int_0^1\sin(st)\cos\left(\frac{a}
{t}\right)\frac{dt}{t}\notag\\
 & = & -\pi\,{\rm sgn}(s)J_0(2\sqrt{\abs{as}}) +2\int_0^1\sin(st)\cos\left(\frac{a}
{t}\right)\frac{dt}{t}.\label{bessel}
\end{eqnarray}
The Bessel function integral appears as \cite[2.7(14)]{erdelyi}.  Since $J_0$
is continuous and $J_0(0)=1$ with $J_0'=-J_1$, the second
distributional derivative is then
$$
\Psif''(s)=-2\pi\delta(s)+\frac{\pi\sqrt{\abs{a}}J_1(2\sqrt{\abs{as}})}{\sqrt{\abs{s}}}
+\fonehat(s).
$$
And,
\begin{equation}
\fhat(s)=\fonehat(s)-\Psif''(s)=2\pi\delta(s)-
\frac{\pi\sqrt{\abs{a}}J_1(2\sqrt{\abs{as}})}{\sqrt{\abs{s}}}.\label{costransform}
\end{equation}
After a change of variables, this assigns a value to the divergent integrals
$$
\int_0^\infty\cos(st)\cos(a/t)\,dt=\int_0^\infty\cos(s/t)\cos(at)\frac{dt}{t^2}
=\pi\delta(s)-
\frac{\pi\sqrt{\abs{a}}J_1(2\sqrt{\abs{as}})}{2\sqrt{\abs{s}}}.
$$

We have not been able to find this transform in the literature.

The sine transform of the function $x\mapsto \sin(a/x)$ converges as a Henstock--Kurzweil
integral \cite[2.7(6)]{erdelyi}.   If we now let $g(x)=e^{ia/x}$ then
$$
\ghat(s)=\intinf\cos(st-a/t)\,dt=2\pi\delta(s)-\left\{\begin{array}{cl}
\frac{2\pi\sqrt{\abs{a}}J_1(2\sqrt{\abs{as}})}{\sqrt{\abs{s}}}, & as<0\\
0, & as\geq 0.
\end{array}
\right.
$$

Let $h(x)=x\sin(a/x)$.
A similar calculation gives
$$
{\hat h}(s)=2\pi a\delta(s)+\frac{\pi aJ_0(2\sqrt{\abs{as}})}{\abs{s}}
-\frac{3\pi \sqrt{\abs{a}}\,{\rm sgn}(a)J_1(2\sqrt{\abs{as}})}{2\abs{s}^{3/2}}.
$$

For $m\in\N$ define $f_m(x)=\cos^m(a/x)$.  The Fourier series of $\cos^m$ 
\cite[1.320]{gradshteyn} with \eqref{costransform} give
\begin{align*}
&{\hat f_{2n-1}}(s)  =  2\pi\delta(s)-\frac{\pi}{2^{2n-2}}\sum_{k=0}^{n-1}\binom{2n-1}{k}
\left|\frac{(2n-2k-1)a}{s}\right|^{1/2}J_1\left(2\left|(2n-2k-1)as\right|^{1/2}\right)\\
&{\hat f_{2n}}(s)  =  2\pi\delta(s)-\frac{\pi}{2^{2n-1}}\sum_{k=0}^{n-1}\binom{2n}{k}
\left|\frac{(2n-2k)a}{s}\right|^{1/2}J_1\left(2\left|(2n-2k)as\right|^{1/2}\right).
\end{align*}
\end{example}

\begin{example}\label{examplearctan}
Define $f(x)=\arctan(x/a)$ for $a\not=0$.  Then $f\in L^\infty(\R)$ but $f$ is not
in any other $L^p(\R)$ space.  The Fourier transform integral diverges when 
considered as a Henstock--Kurzweil integral or
distributional integral \cite{talvilaacrn,  talvilaLp}.  We have
$$
\Psif(s)=\intout e^{-ist}\arctan\left(\frac{t}
{a}\right)\frac{dt}{t^2}=-2i\int_1^\infty\sin(st)\arctan\left(\frac{t}
{a}\right)\frac{dt}{t^2}.
$$
The derived integral converges uniformly for all $\abs{s}\geq\eta>0$ so that
\begin{eqnarray*}
\Psif'(s) & = & -2i\int_0^\infty\cos(st)\arctan\left(\frac{t}
{a}\right)\frac{dt}{t} +2i\int_0^1\cos(st)\arctan\left(\frac{t}
{a}\right)\frac{dt}{t}
\\
 & = & -i\pi\,{\rm sgn}(a)\int_{\abs{as}}^\infty e^{-t}\frac{dt}{t}
+2i\int_0^1\cos(st)\arctan\left(\frac{t}
{a}\right)\frac{dt}{t}.
\end{eqnarray*}
See \cite[1.8(3)]{erdelyi}.  Then
$$
\Psif''(s)=\frac{i\pi\,{\rm sgn}(a)e^{-\abs{as}}}{s} 
-2i\int_0^1\sin(st)\arctan\left(\frac{t}
{a}\right)\,dt.
$$
And, for $a\not=0$,
$$\fhat(s)=-\frac{i\pi\,{\rm sgn}(a)e^{-\abs{as}}}{s}.
$$
This transform was computed in a different manner in \cite{talvilaSlovaca}.

Now let $g(x)=\arctan(a/x)=(\pi/2){\rm sgn}(ax)-\arctan(x/a)$.   Then
$$
\ghat(s)=-\frac{i\pi\,{\rm sgn}(a)}{s}+\frac{i\pi\,{\rm sgn}(a)e^{-\abs{as}}}{s}
=-\frac{2i\pi}{\abs{s}}e^{-\abs{as}/2}\sinh((as)/2).
$$
\end{example}

\section{Inversion}\label{sectioninversion}

A classical inversion theorem states that if $f,\fhat\in L^1(\R)$ then
\begin{equation}
f(x)=\frac{1}{2\pi}\intinf e^{isx}\fhat(s)\,ds \text{ for each } x\in\R.
\label{inversionformula}
\end{equation}

In general, there is no such inversion formula for tempered distributions or for 
functions in $L^p(\R)$ when $p>2$.  A special
case is that of distributions with \L ojasiewicz point values.
See \cite[II.5]{pilipovicstankovicvindas} for interpretation of this formula in terms of
\L ojasiewicz point values of distributions and summability.

In this section a summability kernel (approximate identity) is used to prove an
inversion theorem. 
Pointwise inversion holds at Lebesgue points and thus almost
everywhere.  Inversion holds within the uniform norm when $f$ is uniformly
continuous on $\R$. 

\begin{theorem}\label{theoreminversion}
Let $f\in L^\infty(\R)$.  Let $\psi\in L^1(\R)$ such that $\intinf\psi(t)\,dt=1$.
Assume (i) $\psihat$ is absolutely continuous,
(ii) $\psihat\in L^1(\R)$, (iii) $\psihat\,'\in L^1(\R)$ and (iv) $\psihat\,'$ is of bounded variation.
Write $e_x(t)=e^{xt}$.  For each $a>0$ let $\psi_a(x)=\psi(x/a)/a$.
Define $I_a[f](x)=(1/(2\pi))\intinf e_{ix}\psiahat\fhat$.\\
(a) If $f$ is uniformly continuous on $\R$ then
$\lim_{a\to 0^+}\norm{f-I_a[f]}_\infty=0$.\\
(b) If there are $c>0$ and $\delta>1$ such that 
$\abs{\psi(x)}\leq c(\abs{x}+1)^{-\delta}$ for all $x\in\R$ then
$\lim_{a\to 0^+}I_a[f](x)=f(x)$ at each Lebesgue point of $f$, in particular,
for almost all $x$ and at each point of continuity of $f$.
\end{theorem}
\begin{proof}
The conditions of Theorem~\ref{theoremexchange} also apply to $\psiahat$ and
hence to 
the function $x\mapsto e^{ixt}\psiahat(x)$.  The exchange theorem then gives
$I_a[f](x)=\langle \fhat,e_{ix}\psiahat\rangle/(2\pi)=
\langle f,\widehat{e_{ix}\psiahat}\rangle/(2\pi)$.  
From
\eqref{inversionformula} then 
$I_a[f](x)=f\ast \psi_a$.  Part (a) follows from \cite[Theorem~8.14]{folland}
and part (b) follows from \cite[Theorem~8.15]{folland}.
\end{proof}
\begin{corollary}\label{corollaryinversion}
Let $p_n(x)=x^n$.
If $\psi\in C^1(\R)$ such
that $\psi'$ is absolutely continuous, $p_2\psi\in L^1(\R)$, $\intinf\psi(t)\,dt=1$,
 $p_2\psi'\in L^1(\R)$
and $p_2\psi''\in L^1(\R)$, 
then $\psi$ satisfies the hypotheses of the Theorem and the inversion follows.
\end{corollary}
\begin{proof}
We need to prove that $\psihat$ satisfies conditions (i)-(iv) in the Theorem.
Note conditions ($\alpha$)-($\delta$) in Theorem~\ref{theoremexchange} that guarantee
a Fourier transform is in $L^1(\R)$.

(i) Differentiation under the integral sign gives $\psihat''(s)=-\widehat{p_2\psi}$.  Hence,
$\psihat\in C^2(\R)$ and therefore absolutely continuous.
This also shows that $\psihat$, $\psihat'$ and $\psihat''$ are locally integrable.

(ii) To show $\psihat\in L^1(\R)$ we need, ($\alpha$) $\psi$ absolutely continuous and 
($\beta$) $\psi\in L^1(\R)$.  These follow from the smoothness of $\psi$ and 
$p_2\psi\in L^1(\R)$. ($\gamma$) Since $p_2\psi'\in L^1(\R)$ then $\psi'\in L^1(\R)$.
($\delta$) The variation of $\psi'$ is ${\rm var}(\psi')=\intinf\abs{\psi''(t)}\,dt
<\infty$.

(iii) Show $\psihat'\in L^1(\R)$.  Note that $\psihat'=-i\widehat{p_1\psi}$.
($\alpha$) $p_1\psi$ is absolutely continuous.  ($\beta$) $p_1\psi\in L^1(\R)$.
($\gamma$) $(p_1\psi)'=p_1\psi'+\psi\in L^1(\R)$.
($\delta$) $(p_1\psi)'$ is of bounded variation since $(p_1\psi)''=p_1\psi''+2\psi'\in
L^1(\R)$.

(iv) Show $\psihat'$ is of finite variation.  Since $\psihat\in C^2(\R)$ this follows
if $\psihat''=-\widehat{p_2\psi}\in L^1(\R)$. ($\alpha$) $p_2\psi$ is absolutely continuous.  
($\beta$) $p_2\psi\in L^1(\R)$.
($\gamma$) $(p_2\psi)'=p_2\psi'+2p_1\psi\in L^1(\R)$.
($\delta$) $(p_2\psi)'$ is of bounded variation since $(p_2\psi)''=p_2\psi''+4p_1\psi'
+2\psi\in
L^1(\R)$.
\end{proof}
\begin{example}\label{examplekernels}
The three commonly used summability kernels,\\
Fej\'er--Ces\`aro:\quad $\psi_a(t)=2a\sin^2[t/(2a)]/(\pi t^2)$ with
${\widehat \psi_a}(s)=(1-a\abs{s})\chi_{[-1/a,1/a]}(s)/(2\pi)$,\\
Poisson--Abel:\quad $\psi_a(t)=2a/[\pi(t^2+a^2)]$ with 
${\widehat \psi_a}(s)=e^{-a\abs{s}}/\pi$,\\ 
Weierstrass--Gauss:\quad $\psi_a(t)=e^{-t^2/(4a^2)}/(2\sqrt{\pi}a)$ with
${\widehat\psi_a}(s)=e^{-a^2s^2}/(2\pi)$,\\
satisfy the conditions of Theorem~\ref{theoreminversion}.
While,\\
Dirichlet kernel:\quad $\psi_a(t)=\sin(t/a)/(\pi t)$ with
 ${\widehat\psi_a}(s)=\chi_{[-1/a,1/a]}(s)/(2\pi)$\\
does not.

Of these, only the Weierstrass kernel satisfies the hypotheses of 
Corollary~\ref{corollaryinversion}.
\end{example}

\section{Convolution}\label{sectionconvolution}

A classical result is that if $f,g\in L^1(\R)$ then $\widehat{f\ast g}=\fhat\ghat$.
The exchange theorem leads to a type of weak version of this result,
$\intinf \widehat{f\ast g}h=\intinf \fhat\ghat\,h$ and
$\intinf \fhat g\ast h=\intinf f(t)\ghat(t)\hhat(t)\,dt$, where $f\in L^\infty(\R)$.
As with theorems of this type, roughly speaking, if a proof requires less smoothness
of $g$ it requires more smoothness of $h$, and vice versa.  There are then various
versions of such theorems.  In the versions given, we have tried to 
choose the simplest sets of hypotheses.
Proving the theorems requires finding sufficient conditions for a Fourier transform
to be in $L^1(\R)$.  This is done using Theorem~\ref{theoremexchange}.

Convolution theorems for distributions typically have severe restrictions on supports.
For example, \cite{folland, friedlanderjoshi, zemanian}.
The theorems below have no restriction on supports.

\begin{theorem}\label{theoremconvolution1}
Let $f\in L^\infty(\R)$. Write $p_n(x)=x^n$.  Let $g$ be
absolutely continuous such that
(a) $g'\in L^1(\R)$,
(b) $g'$ is of bounded variation,
(c) $p_2g\in L^1(\R)$.
Let $h\in L^1(\R)\cap C^1(\R)$ such that $h$ and $h'$ are of bounded variation
and $h'$ is absolutely continuous.
Then $\intinf \widehat{f\ast g}h=\intinf\fhat\ghat\,h$.
\end{theorem} 
\begin{proof}
Use the notation of Definition~\ref{defnFT}.  

Since $g$ is absolutely continuous, (c) implies $g\in L^1(\R)$.  With
$f_1,g,h\in L^1(\R)$ we then have 
$\intinf \widehat{f_1\ast g}h=\intinf\fonehat\ghat\,h$.

It is then only
necessary to prove the formula for $f_2$.

With $f_2\in L^\infty(\R)$ we have $f_2\ast g\in L^\infty(\R)$.  Since
$h\in C^1(\R)$ the variation of $h$ is given by $\intinf\abs{h'(t)}\,dt<\infty$.  Then
$h$ satisfies the conditions of Theorem~\ref{theoremexchange} and $\hhat\in L^1(\R)$.
Using the
Fubini--Tonelli theorem,
\begin{eqnarray}
\intinf \widehat{f_2\ast g}h & = & \intinf f_2\ast g(t)\hhat(t)\,dt
 =  \intinf\int_{\abs{s}>1} f_2(s)g(t-s)\hhat(t)\,ds\,dt\notag\\
 & = & \int_{\abs{s}>1} f_2(s)\intinf g(t-s)\hhat(t)\,dt\,ds.\label{conv1}
\end{eqnarray}

Now show the conditions of Theorem~\ref{theoremexchange} apply to $\ghat h$.
Note that $p_1g\in L^1(\R)$.  
Differentiating
under the integral sign then shows $\ghat\in C^2(\R)$.  Hence, $\ghat h$ is absolutely
continuous.  Note that $\intinf\abs{\ghat(s)h(s)}\,ds\leq\norm{\ghat}_\infty\norm{h}_1
\leq\norm{g}_1\norm{h}_1$. And,
\begin{eqnarray*}
\norm{(\ghat h)'}_1 & = & 
\intinf\abs{\ghat'(s)h(s)+\ghat(s)h'(s)}\,ds
 \leq  \norm{\ghat'}_\infty\norm{h}_1+\norm{\ghat}_\infty{\rm var}(h)\\
   & \leq &  \norm{p_1g}_1\norm{h}_1+\norm{g}_1{\rm var}(h).
\end{eqnarray*}
Since $\ghat\in C^2(\R)$ and $h'$ is absolutely continuous, then
\begin{eqnarray*}
{\rm var}[(\ghat h)'] & = & \intinf\abs{\ghat''(s)h(s)+2\ghat'(s)h'(s)+\ghat(s)h''(s)}\,ds\\
 & \leq & \norm{\ghat''}_\infty\norm{h}_1+2\norm{\ghat'}_\infty\norm{h'}_1+
\norm{\ghat}_\infty\norm{h''}_1\\
 & \leq & \norm{p_2g}_1\norm{h}_1+2\norm{p_1g}_1{\rm var}(h)+\norm{g}_1{\rm var}(h').
\end{eqnarray*}
By Theorem~\ref{theoremexchange} then, $\intinf\ftwohat\ghat h=\intout f_2(t)\widehat{\ghat h}(t)
\,dt$.  

Since $g,h\in L^1(\R)$, we get 
$$
\widehat{\ghat h}(t)=\intinf g(u-t)\intinf e^{-ius}h(s)\,ds\,du = \intinf g(u-t)\hhat(u)\,du,
$$
by the Fubini--Tonelli theorem.

Now we have
$$
\intinf\ftwohat\ghat\,h  = 
\intout f_2(t)\widehat{\ghat h}(t)
\,dt
  =  \intout f_2(t) \intinf g(u-t)\hhat(u)\,du\,dt,
$$
in agreement with \eqref{conv1}.
\end{proof}

\begin{theorem}\label{theoremconvolution2}
Let $f\in L^\infty(\R)$.  Let $g\in L^1(\R)$. Let $h\in C^1(\R)$ such that
$h'$ is absolutely continuous and
$h,h',h''\in L^1(\R)$.  Then $\intinf \fhat g\ast h=\intinf f(t)\ghat(t)\hhat(t)\,dt$.
\end{theorem}
\begin{proof}
Use the notation of Definition~\ref{defnFT}.  

Since $f_1,g,h\in L^1(\R)$ we have 
$\intinf \fonehat g\ast h=\intin f(t)\ghat(t)\hhat(t)\,dt$.  It is then only
necessary to prove the formula for $f_2$.

Show that $g\ast h$ satisfies the conditions of Theorem~\ref{theoremexchange}.  
Differentiation under the integral sign gives
$(g\ast h)'(s) =g\ast h'(s)$, which shows that $g\ast h\in C^1(\R)$ and hence 
absolutely continuous.  And, $\norm{g\ast h}_1\leq \norm{g}_1\norm{h}_1$.  Similarly,
$\norm{(g\ast h)'}_1\leq \norm{g}_1\norm{h'}_1$.  The variation of $(g\ast h)'$
is given by ${\rm var}[(g\ast h)']=\intinf\abs{(g\ast h)''(s)}\,ds\leq
\norm{g}_1\norm{h''}_1$.

The exchange theorem then gives
$$
\intinf \ftwohat g\ast h=\intout f(t)\widehat{g\ast h}(t)\,dt=\intout f(t)\ghat(t)\hhat(t)\,dt.
$$
\end{proof}


\begin{thebibliography}{99}
\bibitem{abdelhakim}
A.A. Abdelhakim, {\it On the unboundedness in $\Lany{q}$ of the
Fourier transform of $\Lany{p}$ functions},
arXiv:1806.03912, 2020.
\bibitem{appell}
J. Appell, J. Banas and N.J. Merentes D\'{i}az, {\it Bounded variation and around},
Berlin, De Gruyter, 2014.
\bibitem{donoghue}
W.F. Donoghue, Jr., {\it Distributions and Fourier transforms},
New York, Academic Press, 1969.
\bibitem{erdelyi}
A. Erd\'elyi (Ed.), {\it Tables of integral transforms, vol. I},
New York, McGraw-Hill, 1954.
\bibitem{folland}
G.B. Folland, {\it Real analysis}, New York, Wiley, 1999.
\bibitem{friedlanderjoshi}
F.G. Friedlander and M. Joshi,
{\it Introduction to the theory of distributions},
Cambridge, Cambridge University Press, 1999.
\bibitem{gradshteyn}
I.S. Gradshteyn and I.M. Ryzhik,
{\it Table of integrals, series and products}
(trans. Scripta Technica, Inc., eds.
A. Jeffrey and D. Zwillinger), San Diego,
Academic Press, 2015.
\bibitem{grafakosclassical}
L. Grafakos, {\it Classical Fourier analysis},
New York, Springer, 2008.
\bibitem{grafakosmodern}
L. Grafakos, {\it Modern Fourier analysis},
New York, Springer, 2009.
\bibitem{kolmogorov1926}
A.N. Kolmogorov,  {\it Une s\'erie de Fourier--Lebesgue divergente partout},
C. R. Acad. Sci. Paris {\bf 183}(1926), 1327--1328.
\bibitem{liebloss}
E.H. Lieb and M. Loss, {\it Analysis}, Providence, American Mathematical
Society, 2001.
\bibitem{mcleod}
R.M. McLeod, {\it The generalized Riemann integral}, Washington,
The Mathematical Association of America, 1980.
\bibitem{pilipovicstankovicvindas}
S. Pilipovi\'c, B. Stankovi\'c and J. Vindas,
{\it Asymptotic behavior of generalized functions},
Singapore, World Scientific, 2011.
\bibitem{steinweiss}
E.M. Stein and G. Weiss, {\it Introduction to Fourier analysis on Euclidean spaces},
Princeton, Princeton University Press, 1971.
\bibitem{talviladenjoy}
E.  Talvila, {\it  The distributional Denjoy  integral}, Real Anal. Exchange
{\bf 33}(2008), 51--82.
\bibitem{talvilaacrn}
E. Talvila, {\it Integrals and Banach spaces for finite order distributions},
Czechoslovak Math. J. {\bf 62}(137) (2012), no. 1, 77--104.
\bibitem{talvilaLp}
E. Talvila, {\it The $L^p$ primitive integral}, Math. Slovaca {\bf 64}(2014), 1497--1524.
\bibitem{talvilafourieramm}
E. Talvila, {\it Fourier transform inversion using an elementary differential equation and a contour integral},
Amer. Math. Monthly  {\bf 126}(2019), 717--727.
\bibitem{talvilaSlovaca}
E. Talvila,
{\it Fourier transform inversion: Bounded variation, polynomial growth, 
Henstock-Stieltjes integration},
Math. Slovaca {\bf 73}(2023), no. 1, 131--146.
\bibitem{talvilaLpFourier}
E. Talvila, {\it The Fourier transform in Lebesgue spaces},
Czechoslovak Math. J. (2024) 
\url{https://doi.org/10.21136/CMJ.2024.0001-23}.
\bibitem{zemanian}
A.H. Zemanian, {\it Distribution theory and transform analysis},
New York, Dover, 1987.
\end{thebibliography}
\end{document}